\documentclass[11pt,a4paper]{article}
\usepackage[utf8]{inputenc}
\usepackage{amsmath}
\usepackage{amsthm}
\usepackage{amsfonts}
\usepackage{amssymb}
\author{Mark Shusterman}
\title{Groups with positive rank gradient \\ and their actions}
\usepackage{mathtools}
\usepackage{graphicx}
\usepackage{hyperref}
\newtheorem{theorem}{Theorem}[section]
\newtheorem{lemma}[theorem]{Lemma}
\newtheorem{proposition}[theorem]{Proposition}
\newtheorem{corollary}[theorem]{Corollary}
\numberwithin{equation}{section}

\newcommand{\lemref}[1]{\hyperref[#1]{Lemma \ref*{#1}}}
\newcommand{\thmref}[1]{\hyperref[#1]{Theorem \ref*{#1}}}
\newcommand{\propref}[1]{\hyperref[#1]{Proposition \ref*{#1}}}
\newcommand{\corref}[1]{\hyperref[#1]{Corollary \ref*{#1}}}

\makeatletter
\def\moverlay{\mathpalette\mov@rlay}
\def\mov@rlay#1#2{\leavevmode\vtop{%
   \baselineskip\z@skip \lineskiplimit-\maxdimen
   \ialign{\hfil$\m@th#1##$\hfil\cr#2\crcr}}}
\newcommand{\charfusion}[3][\mathord]{
    #1{\ifx#1\mathop\vphantom{#2}\fi
        \mathpalette\mov@rlay{#2\cr#3}
      }
    \ifx#1\mathop\expandafter\displaylimits\fi}
\makeatother

\makeatletter
\newcommand*{\defeq}{\mathrel{\rlap{%
                     \raisebox{0.27ex}{$\m@th\cdot$}}%
                     \raisebox{-0.27ex}{$\m@th\cdot$}}%
                     =}
\makeatother

\begin{document}

\maketitle

\abstract{ We show that given a finitely generated LERF group $G$ with positive rank gradient, and finitely generated subgroups $A,B \leq G$ of infinite index, one can find a finite index subgroup $B_0$ of $B$ such that $[G : \langle A \cup B_0 \rangle] = \infty$. This generalizes a theorem of Olshanskii on free groups. We conclude that a finite product of finitely generated subgroups of infinite index does not cover $G$. We construct a transitive virtually faithful action of $G$ such that the orbits of finitely generated subgroups of infinite index are finite. Some of the results extend to profinite groups with positive rank gradient.}

\section{Introduction}

The rank gradient of a finitely generated group $G$ is defined to be
\begin{equation} \label{DefRGEq}
\nabla G  \defeq \inf_{U} \frac{d(U) - 1}{[G : U]}
\end{equation}
where $U$ ranges over all subgroups of finite index in $G$, and $d(U)$ stands for the smallest cardinality of a generating set of $U$. The notion of rank gradient has been introduced in \cite{L1} and further studied, for instance in \cite{AGN}, \cite{AJN}, \cite{AN}, \cite{KN}, \cite{Os}, \cite{P}, and \cite{Sch}. It is our point of view that many interesting properties of a group (e.g. a free group) can be deduced using only the positivity of its rank gradient, as explained and demonstrated in \cite{Sh}. However, in order to effectively use the rank gradient, we inevitably need to make an additional assumption that will provide us with some finite index subgroups to which we can apply \eqref{DefRGEq}. This is achieved by considering the profinite topology of a group.

We think of all the groups as being topological by endowing them with the profinite topology, i.e. by taking as a basis the cosets of finite index subgroups. In this vein, recall that a group $G$ is LERF (locally extended residually finite), or subgroup separable, if its finitely generated subgroups are closed, or equivalently, if each finitely generated subgroup of $G$ is the intersection of some collection of finite index subgroups of $G$.

Our first result generalizes \cite[Theorem 1.1]{Ol} which is the statement one gets by taking $G$ to be a nonabelian free group in the following.

\begin{theorem} \label{MainThm}

Let $G$ be a finitely generated LERF group with positive rank gradient, and let $A,B$ be finitely generated subgroups of infinite index in $G$. Then there exists a finite index subgroup $H$ of $B$ such that $A$ and $H$ generate a subgroup of  infinite index in $G$.

\end{theorem}  

As in \cite[Theorem 1.1]{Ol}, if we are also given a finite subset $S \subseteq G \setminus A$, then the fact that $G$ is LERF gives us a finite index subgroup $U$ of $G$ which contains $A$ and avoids $S$. By taking $H_0 \defeq H \cap U$ we also assure that $\langle A \cup H_0 \rangle$ avoids $S$. A consequence of \thmref{MainThm} is:

\begin{theorem} \label{SecThm}

Let $G$ be a finitely generated LERF group with positive rank gradient, let $n \in \mathbb{N}$, and let $H_1, \dots, H_n$ be finitely generated subgroups of infinite index in $G$. Then 
\begin{equation}
H_1H_2 \cdots H_n \subsetneq G.
\end{equation}

\end{theorem}

This means that $G$ is not boundedly generated in a rather strong sense, thus improving upon \cite{Sh0} in the LERF case. Another application of \thmref{MainThm} is the construction of 'locally finite' actions. Similar actions of free and hyperbolic groups are constructed in \cite{Ol} and \cite{Ch}.

\begin{theorem} \label{ThirdThm}

Let $G$ be a finitely generated LERF group with $\nabla G >0$. Then there exists a (right) transitive action of $G$ on a set $X$ such that:

\begin{itemize}

\item There are only finitely many $g \in G$ which act trivially on $X$.

\item For every finitely generated subgroup $L$ of infinite index in $G$, and any $x \in X$, the orbit $xL$ is finite.

\end{itemize}

\end{theorem}

This gives us an almost faithful action on an infinite set which is 'locally finite' even though it is transitive. For instance, it follows that for all $g \in G$ and $x \in X$ the set $\{xg^n : n \in \mathbb{Z}\}$ is finite.  

Some examples of groups to which our theorems apply are:

\begin{enumerate}

\item Nonabelian free groups and nonabelian limit groups.

\item Surface groups and nonabelian Fuchsian groups.

\item Free products of finitely generated LERF groups of order $> 2$.

\item Free products of infinite finitely generated LERF groups amalgamating a finite subgroup.

\item Free products of finitely generated nonabelian Fuchsian groups with cyclic amalgamation.

\item Free products of infinite finitely generated nilpotent groups amalgamating a finite cyclic subgroup.

\item Fundamental groups of connected sums of compact hyperbolic $3$-manifolds.

\item Finitely presented LERF groups with deficiency $\geq 2$.

\item Graph groups whose graph is disconnected and does not contain neither an induced path of length $3$, nor an induced square.

\end{enumerate}

Our first two results hold under an assumption weaker than LERF, namely \textbf{LPF}, introduced in \cite[Definition 3.11]{C}. We say that a group $G$ is \textbf{LPF} if every finitely generated subgroup $H$ of infinite index in $G$ is contained in a subgroup $U$ of arbitrarily large finite index in $G$, or equivalently, if the closure of $H$ in $G$ is of infinite index. Furthermore, as explained in \ref{ProGr}, \thmref{MainThm} and \thmref{SecThm} have natural analogues for profinite groups, as all of their subgroups are closed by definition.

\section{Profinite measure}

Let $G$ be a group, and let $\mathcal{P}(G)$ be the family of its subsets. Define
\begin{equation} \label{DefProfMesEq}
\mu \colon \mathcal{P}(G) \to [0,1], \quad \mu(S) \defeq \inf_{\varphi} \frac{|\varphi(S)|}{|\varphi(G)|}
\end{equation}
where $\varphi$ ranges over all the epimorphisms from $G$ onto finite groups. We call $\mu$ the \textbf{profinite measure} on $G$, even though it is not a measure in case that $G$ is infinite. The profinite measure does however enjoy the following properties, the trivial proof of which is omitted.

\begin{enumerate}

\item Monotonicity: for $A \subseteq B \subseteq G$ we have 
\begin{equation} \label{MonotEq}
0 = \mu(\emptyset) \leq \mu(A) \leq \mu(B) \leq \mu(G) = 1.
\end{equation}

\item Subadditivity: for $A,B \subseteq G$ we have 
\begin{equation} \label{SubAddEq}
\mu(A \cup B) \leq \mu(A) + \mu(B).
\end{equation}

\item Translation invariance: for $g,h \in G, \ A \subseteq G$ we have 
\begin{equation} \label{TransInvEq}
\mu(gAh) = \mu(A).
\end{equation}

\end{enumerate}

Since the profinite measure 'considers' only finite images, we see that its value on a set coincides with its value on the closure of the set. Also, if $\widehat{G}$ is the profinite completion of $G$, then the profinite measure of a subset of $G$, is just the Haar measure of its closure in $\widehat{G}$. As we have already mentioned earlier, any closed subgroup $H \leq G$ is the intersection of a family of finite index subgroups of $G$. In light of that, if $[G : H] = \infty$ then $H$ is contained in a subgroup of $G$ with arbitrarily large finite index. We denote by $H_G$ the normal core of $H$ in $G$.

\begin{proposition} \label{MesSubProp}

For a closed subgroup $H$ of a group $G$ we have 
\begin{equation} \label{IndVsMesEq}
\mu(H) = \frac{1}{[G : H]}.
\end{equation}

\end{proposition}

\begin{proof}

First, suppose that $n \defeq [G : H] < \infty$, and let $\{g_1, \dots , g_n\}$ be a right transversal of $H$ in $G$. The following inequalities give us \eqref{IndVsMesEq}: 
\begin{equation} \label{UBoundFin}
\mu(H) \stackrel{\ref{DefProfMesEq}}{\leq} \frac{|H/H_G|}{|G/H_G|} = \frac{[H : H_G]}{[G : H_G]} = \frac{[H : H_G]}{[G : H][H : H_G]}= \frac{1}{n}.
\end{equation}

\begin{equation}
1 \stackrel{\ref{MonotEq}}{=} \mu(G) = \mu(\bigcup_{i=1}^n Hg_i) \stackrel{\ref{SubAddEq}}{\leq} \sum_{i=1}^n \mu(Hg_i) \stackrel{\ref{TransInvEq}}{=} \sum_{i=1}^n \mu(H) = n\mu(H).
\end{equation}

Now, suppose that $[G : H] = \infty$, and take some $\epsilon > 0$. Since $H$ is closed, there exists some $H \leq U \leq G$ with $\frac{1}{\epsilon} \leq [G : U] < \infty$. Hence, 
\begin{equation}
\mu(H) \stackrel{\ref{MonotEq}}{\leq} \mu(U) = \frac{1}{[G : U]} \leq \epsilon.
\end{equation}

\end{proof}

\section{Finitely generated subgroups}

Following \cite{C}, we define the \textbf{proindex} of a subgroup $H$ in a group $G$ to be the supremum of indices of finite index subgroups of $G$ above $H$.
 
\begin{lemma} \label{WeakOlLem}

Let $G$ be a finitely generated group with $\nabla G > 0$, and let $A,B$ be finitely generated subgroups of infinite proindex in $G$. Then there exist finite index subgroups $A_0 \leq A, B_0 \leq B$ which generate a subgroup of infinite index in $G$.

\end{lemma}

\begin{proof}

Since the proindices are infinite, there exist finite index subgroups $U,V \leq G$ containing $A,B$ respectively, such that 
\begin{equation} \label{BigProIndEq}
[G : U],[G : V] \geq \frac{2\max\{d(A), d(B)\}}{\nabla G}.
\end{equation} 
Set $A_0 \defeq A \cap V, \ B_0 \defeq B \cap U, \ C \defeq \langle A_0 \cup B_0 \rangle$, and note that 
\begin{equation} \label{IndBoundEq}
\begin{split}
&[A : A_0] = [A : A \cap U \cap V] \leq [U : U \cap V] \\
&[B : B_0] = [B : B \cap U \cap V] \leq [V : U \cap V].
\end{split}
\end{equation}
Using Schreier's bound on the rank of a finite index subgroup, we get
\begin{equation} \label{BitLongEq}
\begin{split}
d(C) &\leq d(A_0) + d(B_0) \leq [A : A_0]d(A) + [B : B_0]d(B) \\
&\stackrel{\ref{IndBoundEq}}{\leq} [U : U \cap V]d(A) + [V : U \cap V]d(B) \\
&= [G : U \cap V](\frac{d(A)}{[G : U]} + \frac{d(B)}{[G : V]}) \\
&\leq \frac{2[G : U \cap V]\max\{d(A), d(B)\}}{\min\{[G : U],[G : V]\}} \\
&\stackrel{\ref{BigProIndEq}}{\leq} [G : U \cap V] \nabla G.
\end{split}
\end{equation}
Since $A_0,B_0 \leq U \cap V$, we see that $C \leq U \cap V$. Were the index $[G : C]$ finite, we would have the following contradiction to \eqref{BitLongEq}:
\begin{equation}
d(C) \stackrel{\ref{DefRGEq}}{\geq} [G : C]\nabla G + 1 \geq [G : U \cap V]\nabla G + 1.
\end{equation}

\end{proof}

As in \cite[Definition 3.11]{C}, we say that a group $G$ is \textbf{LPF} if the index and proindex coincide for every finitely generated subgroup of $G$.

\begin{corollary} \label{ProdCor}

Let $G$ be a finitely generated LPF group with $\nabla G > 0$, and let $A,B$ be finitely generated subgroups of infinite index in $G$. Then 
\begin{equation}
\mu(AB) = 0.
\end{equation}

\end{corollary}

\begin{proof}

By \lemref{WeakOlLem}, there exist finite index subgroups $A_0,B_0$ of $A,B$ respectively, such that the index of $C \defeq \langle A_0 \cup B_0 \rangle$ in $G$ is infinite. Since $A_0$ and $B_0$ are finitely generated, $C$ is finitely generated as well, so its proindex in $G$ is infinite as $G$ is LPF. Therefore, by \propref{MesSubProp},
\begin{equation} \label{VanishMesEq}
\mu(C) = 0.
\end{equation}
Taking $L$ to be a left transversal of $A_0$ in $A$, and $R$ to be a right transversal of $B_0$ in $B$, we get 

\begin{equation}
\begin{split}
\mu(AB) &= \mu(\bigcup_{L,R} \ell A_0B_0 r) \stackrel{\ref{SubAddEq}}{\leq} \sum_{L,R} \mu(\ell A_0B_0 r)
\stackrel{\ref{TransInvEq}}{=} \sum_{L,R} \mu(A_0B_0) \\
&= |L||R|\mu(A_0B_0) \stackrel{\ref{MonotEq}}{\leq} |L||R|\mu(C) \stackrel{\ref{VanishMesEq}}{=} 0.
\end{split}
\end{equation}

\end{proof}

We need a simple observation for the proof of \thmref{MainThm}.

\begin{proposition} \label{ObsProp}

Let $\varphi \colon G \to K$ be a group homomorphism, let $N$ be its kernel, and let $A,B \subseteq G$. Then
$\varphi(A) \cap \varphi(B) \subseteq \varphi(B \cap NA).$

\end{proposition}

\begin{proof}

Let $z \in \varphi(A) \cap \varphi(B)$. There exist $a \in A, \ b \in B$ such that $z = \varphi(a) = \varphi(b)$. Thus, $b \in Na \subseteq NA$, so $z = \varphi(b) \in \varphi(B \cap NA)$.

\end{proof}

Let us now prove \thmref{MainThm}.

\begin{theorem} \label{OlThm}

Let $G$ be a finitely generated LPF group with $\nabla G > 0$, and let $A,B$ be finitely generated subgroups of infinite index in $G$. Then there is a finite index subgroup $B_0 \leq B$ such that $[G : \langle A \cup B_0 \rangle] = \infty$.

\end{theorem}

\begin{proof}

Set 
\begin{equation} \label{ProofDef}
r \defeq \max\{d(A), d(B)\}, \quad \epsilon \defeq \frac{\nabla G}{2r}.
\end{equation}
By \corref{ProdCor}, $\mu(AB) = 0$, so there exists an epimorphism onto a finite group $\varphi \colon G \to K$, such that 
\begin{equation} \label{MesZerABEq}
\frac{|\varphi(AB)|}{|K|} \stackrel{\ref{DefProfMesEq}}{\leq} \epsilon.
\end{equation}
Put 
\begin{equation} \label{ExtraDefEq}
N \defeq \mathrm{Ker}(\varphi), \quad  B_0 \defeq B \cap NA, \quad C \defeq \langle A \cup B_0 \rangle
\end{equation}
and note that $B \cap N = B_0 \cap N$, so
\begin{equation} \label{CompareIndEq}
\begin{split}
[B : B_0] &= \frac{[B : B \cap N]}{[B_0 : B \cap N]} = \frac{[B : B \cap N]}{[B_0 : B_0 \cap N]} \stackrel{\ref{ExtraDefEq}}{=} \frac{|\varphi(B)|}{|\varphi(B_0)|} \\ &\stackrel{\ref{ObsProp}}{\leq} \frac{|\varphi(B)|}{|\varphi(A) \cap \varphi(B)|}
= \frac{|\varphi(A)\varphi(B)|}{|\varphi(A)|} = \frac{|\varphi(AB)|}{|\varphi(A)|} \\ &\stackrel{\ref{MesZerABEq}}{\leq}
\frac{\epsilon|K|}{|\varphi(A)|} = \epsilon[K : \varphi(A)] \stackrel{\ref{ExtraDefEq}}{=} \epsilon[G : NA].
\end{split}
\end{equation}
Applying Schreier's bound we see that 
\begin{equation} \label{RankUpBoundddEq}
\begin{split}
d(C) &\stackrel{\ref{ExtraDefEq}}{\leq} d(A) + d(B_0) \leq d(A) + [B : B_0]d(B) \\
&\stackrel{\ref{ProofDef}}{\leq} r(1+[B:B_0]) \leq 2r[B:B_0] \stackrel{\ref{CompareIndEq}}{\leq} 2r\epsilon [G : NA] \\
&\stackrel{\ref{ProofDef}}{=} [G : NA]\nabla G.
\end{split}
\end{equation}
If $[G : C]$ were finite, we would get a contradiction to \eqref{RankUpBoundddEq}: 
\begin{equation}
d(C) \stackrel{\ref{DefRGEq}}{\geq} [G : C]\nabla G + 1 \stackrel{\ref{ExtraDefEq}}{\geq} [G : NA]\nabla G + 1. 
\end{equation}

\end{proof}

\section{Corollaries}

We give some corollaries of \thmref{OlThm}, the first of which is a strong form of \thmref{SecThm}.

\begin{corollary} \label{NegCor}

Let $G$ be a finitely generated LPF group with $\nabla G > 0$, let $n \in \mathbb{N}$, and let $H_1, \dots, H_n$ be finitely generated subgroups of infinite index in $G$. Then 
\begin{equation}
\mu(H_1 \cdots H_n) = 0.
\end{equation}

\end{corollary}

\begin{proof}

We induct on $n$, using \propref{MesSubProp} to see that the base case $n=1$ holds. For $n \geq 2$, \thmref{OlThm} gives us a finitely generated subgroup $C$ of infinite index in $G$ which contains both $H_{n-1}$ and a finite index subgroup $H$ of $H_n$. By induction,
\begin{equation} \label{IndVanEq}
\mu(H_1 \cdots H_{n-2}C) = 0
\end{equation}
so by taking $R$ to be a right transversal of $H$ in $H_n$, we see that
\begin{equation}
\begin{split}
\mu(H_1 \cdots H_{n-2} H_{n-1} H_n) &= \mu(H_1 \cdots H_{n-2}H_{n-1}HR) \\
&\stackrel{\ref{MonotEq}}{\leq} \mu(H_1 \cdots H_{n-2}\langle H_{n-1} \cup H \rangle R) \\
&\stackrel{\ref{MonotEq}}{\leq} \mu(H_1 \cdots H_{n-2}CR) \\
&= \mu(\bigcup_{r \in R} H_1 \cdots H_{n-2} C r) \\
&\stackrel{\ref{SubAddEq}}{\leq} \sum_{r \in R} \mu(H_1 \cdots H_{n-2} C r) \\
&\stackrel{\ref{TransInvEq}}{=} |R|\mu(H_1 \cdots H_{n-2}C) \stackrel{\ref{IndVanEq}}{=} 0. 
\end{split}
\end{equation}

\end{proof}

In order to prove \thmref{ThirdThm}, we follow the argument in the proof of \cite[Corollary 4.1 (a)]{Ol}, and instead of using \cite[Theorem 1.1]{Ol}, we invoke \thmref{OlThm}. The conclusion is that any LPF group $G$ with $\nabla G > 0$ contains an infinite index subgroup $R$ such that every finitely generated subgroup $L$ of infinite index in $G$ contains $L \cap R$ as a finite index subgroup. In other words, $G$ acts transitively on $X \defeq R \backslash G$ such that for every $x \in X$ and every finitely generated subgroup $L$ of infinite index in $G$, the orbit $xL$ is finite (see the proof of \cite[Corollary 4.5 (a)]{Ol} for a detailed explanation). Hence, in order to establish \thmref{ThirdThm}, we only need to show that the kernel of the action (those $g$ in $G$ which act trivially) is finite. For that, recall that $G$ is said to be \textbf{RF} (residually finite) if for every $K \subseteq G$ with $|K| \geq M \in \mathbb{R}$, there exists a finite index subgroup $U \lhd G$ such that $|KU/U| \geq M$. If $G$ is LERF, it is also RF.

\begin{theorem}

Let $G$ be a finitely generated residually finite group with $\nabla G > 0$, and let $R$ be an infinite index subgroup of $G$ such that for every finitely generated subgroup $L$ of infinite index in $G$ we have $[L : L \cap R] < \infty$. Then $K \defeq \mathrm{Ker}(R \backslash G \curvearrowleft G)$ is finite.

\end{theorem}

\begin{proof}

Towards a contradiction, suppose that $|K| \geq \frac{d(G)}{\nabla G}$. Since $G$ is residually finite, there exists a finite index subgroup $U \lhd G$ such that
\begin{equation} \label{KSizeLBoundEq}
|KU/U| \geq \frac{d(G)}{\nabla G}. 
\end{equation}
For every finite index subgroup $V \leq U$ we have
\begin{equation} \label{UpBSRGFFEq}
\begin{split}
d(U/U \cap K) &= d(UK/K) \leq d(UK) \leq d(G)[G : UK] \\
&= \frac{d(G)[G : U]}{[UK : U]} \leq \frac{d(G)[G : V]}{[KU : U]} \\ 
&\stackrel{\ref{DefRGEq}}{\leq} \frac{d(G)(d(V) - 1)}{\nabla G [KU : U]} \stackrel{\ref{KSizeLBoundEq}}{\leq} d(V)-1.
\end{split}
\end{equation}
Let $T$ be a generating set of $U/U \cap K$ of minimal cardinality, let $S \subseteq U$ be a set mapped bijectively to $T$ by the quotient map $q \colon U \to U/U \cap K$, and set $L \defeq \langle S \rangle.$ Clearly, $L \leq U$ and $L(U \cap K) = U$ since $T \subseteq q(L)$. Were the index of $L$ in $U$ finite, we would get a contradiction as follows:
\begin{equation}
d(L) \leq |S| = |T| = d(U/U \cap K) \stackrel{\ref{UpBSRGFFEq}}{\leq} d(L)-1.
\end{equation}
Hence, $[U : L] = \infty$. By our assumption, there exists a finite left transversal $F$ of $L \cap R$ in $L$. As $K \leq R$, it follows that
\begin{equation}
U = L(U \cap K) \subseteq L(U \cap R) = F(L \cap R)(U \cap R) \subseteq F(U \cap R)
\end{equation}
so $[U : U \cap R] < \infty$ and thus $[G : R] = \frac{[G:U][U : U \cap R]}{[R : U \cap R]} < \infty$ - an absurdity. 

\end{proof}

\section{Profinite groups} \label{ProGr}

Since our arguments are focused mainly on finite index subgroups, some of our results are more naturally stated for profinite groups. For these groups, only closed subgroups are considered, so additional separability assumptions such as RF, LPF, LERF are not required.

\begin{theorem} \label{MainProfThm}

Let $\Gamma$ be a finitely generated profinite group with positive rank gradient, and let $A,B$ be finitely generated subgroups of infinite index in $\Gamma$. Then there exists an open subgroup $B_0$ of $B$ such that $A$ and $B_0$ generate a subgroup of infinite index in $\Gamma$.

\end{theorem}

The proof is identical to that of \thmref{OlThm}.

\begin{theorem}

Let $\Gamma$ be a finitely generated profinite group with positive rank gradient, let $n \in \mathbb{N}$, and let $H_1, \dots, H_n$ be finitely generated subgroups of infinite index in $\Gamma$. Then 
\begin{equation}
\mu_{\Gamma}(H_1 \cdots H_n) = 0.
\end{equation}

\end{theorem}

Here $\mu_{\Gamma}$ stands for the Haar measure on $\Gamma$. The proof follows that of \corref{NegCor}. Some profinite groups with positive rank gradient are:

\begin{enumerate}

\item Nonabelian free profinite, free pro-$p$, and free prosolvable groups.

\item Free pro-$p$ products.

\item Groups satisfying Schreier's formula.

\item Nonsolvable Demushkin groups and surface groups.

\item Pro-$p$ groups with deficiency at least $2$.

\item Pro-$p$ groups from the class $\mathcal{L}$ all of whose abelian subgroups are procyclic.

\item Completions of groups from the list in the introduction.

\end{enumerate}

\section*{Acknowledgments}

I would like to thank Alexander Olshanskii, Yves de Cornulier, Ashot Minasyan, Benjamin Steinberg, and Henry Wilton for many helpful discussions. This research was partially supported by a grant of the Israel Science Foundation with cooperation of UGC no. 40/14.

\item Author's address: Open Space, Schreiber Building (Mathematics), Tel-Aviv University, Levanon Street, Tel-Aviv, Israel.

\item Author's email: markshus@mail.tau.ac.il

\item Author's website: \url{markshus.wix.com/math}

\end{document}